\documentclass[12pt,reqno]{amsart}

\usepackage[dvips]{graphicx}
\usepackage[arrow,matrix,curve]{xy} 

\usepackage{amssymb, latexsym, amsmath, amscd, array, %makeidx
}

\newtheorem{theorem}{Theorem}[section]
\newtheorem{lemma}[theorem]{Lemma}
\newtheorem{proposition}[theorem]{Proposition}

\newtheorem{conjecture}[theorem]{Conjecture}

\theoremstyle{definition}
\newtheorem{definition}[theorem]{Definition}

\newtheorem{remark}[theorem]{Remark}

\numberwithin{equation}{section}
\numberwithin{figure}{section} 
\numberwithin{table}{section}

\DeclareMathOperator{\PD}{{\rm PD}}

\DeclareMathOperator{\disp}{{\rm disp}}

\DeclareMathOperator{\sys}{{\rm sys}}

\DeclareMathOperator{\area}{{\rm area}}

\newcommand{\Mob}{{\rm Mob}}

\newcommand\dist{{\rm dist}}

\newcommand \arc{\, {\rm arc}}

\newcommand \cqfd{\unskip\kern 6pt\penalty 500
\raise -2pt\hbox{\vrule\vbox to10pt{\hrule width 4pt
\vfill\hrule}\vrule}\par}                 

\def\adots{\mathinner{\mkern2mu\raise1pt\hbox{.}
\mkern3mu\raise4pt\hbox{.}\mkern1mu\raise7pt\hbox{.}}}
\def\hfl#1{\frac{\buildrel{#1}}{{\hbox to 12mm{\rightarrowfill}}}}
  
\def \\R^n \times \R^n
\rightarrow \R{\mathop{\R^n \times \R^n
\rightarrow \R}}

\newcommand\C{{\mathbb {C}}}

 \newcommand\R {{\mathbb R}}
\newcommand\RR {{\mathbb R}} \newcommand\RP {{\mathbb R}{\mathbb P}}
\newcommand\T {{\mathbb T}}

\newcommand{\smat}[4]
{{\(\!\!\begin{array}{cc}{#1}\!&\!{#2}\\begin{equation*}-0.1cm]{#3}\!&\!{#4}\end{array}\!\!\)}}

\newcommand{\rp}{\mathbb R\mathbb P}

\long\def\forget#1\forgotten{} %
\long\def\forgett#1\forgottent{} %
\def\circ{\mathchoice%
 {\mathrel{\raise 1pt\hbox{$\scriptstyle\mathchar"020E$}}}
 {\mathrel{\raise 1pt\hbox{$\scriptstyle\mathchar"020E$}}}
 {\mathrel{\raise 1pt\hbox{$\scriptscriptstyle\mathchar"020E$}}}
 {}
}

\newcommand{\nc}{\newcommand} \nc{\on}{\operatorname}

\nc{\df}{\on{\it df}}

\nc{\conf}{\on{conf}}

\nc{\spt}{\on{spt}}
\nc{\norm}[1]{\| #1 \|}
\nc{\parallelleer}{\norm{\ }} %{\parallel \: \parallel}
\nc{\parallelh}{\norm h} %{\parallel\!h\!\parallel}
\nc{\parallelk}{\norm k} %{\parallel\!k\!\parallel}
\nc{\parallelx}{\norm x} %{\parallel\!x\!\parallel}
\nc{\parallelhrr}{\norm {h_\RR}} %{\parallel\!h_{\rr}\!\parallel}
\nc{\parallelom}{\norm \omega} %{\parallel\!\omega\!\parallel}
\nc{\parallelomij}{\norm {\omega_{i_j}}} %{\parallel\!\omega_{i_j}\!\parallel}
\nc{\parallelomx}{\norm {\omega_{x}}} %{\parallel\!\omega_{x}\!\parallel}
\nc{\parallelpi}{\norm \pi} %{\parallel\!\pi\!\parallel}
\nc{\parallelalf}{\norm \alpha} %{\parallel\!\alpha\!\parallel}
\nc{\parallelalfs}{\norm {\alpha_s}} %{\parallel\!\alpha_s\!\parallel}
\nc{\parallelalfi}{\norm {\alpha_i}} %{\parallel\!\alpha_i\!\parallel}
\nc{\parallelalfij}{\norm {\alpha_{i_j}}} %{\parallel\!\alpha_{i_j}\!\parallel}
\nc{\parallelbeta}{\norm \beta} %{\parallel\!\beta\!\parallel}
\nc{\parallelbetat}{\norm {\beta_t}} %{\parallel\!\beta_t\!\parallel}
\nc{\parallelhcapalf}{\norm {h \cap \alpha}} %{\parallel\!h\cap\alpha\!\parallel}
\nc{\parallelPDralf}{\norm {\PD_\RR(\alpha)}} %{\parallel\!\PD_{\rr}(\alpha)\!\parallel}
\nc{\strichleer}{| \  |}
\nc{\NN}{\mathbb N}

\nc{\rr}{\mbox{$\scriptstyle\mathbb R$}}

\nc{\dF}{{\it dF}} 

\nc{\DF}{{\it DF}} 

\nc{\ds}{{\it ds}} 

\nc{\dvol}{{\it dvol}}

\nc{\grad}{{\rm grad}} 
\nc{\strichw}{\|\omega\|} 
\nc{\strichwx}{|\omega_x|}
\nc{\Hess}{{\rm Hess}}

\begin{document}

\title[Hyperellipticity and Klein bottle
companionship]{Hyperellipticity and Klein bottle companionship in
systolic geometry}

\author{Karin Usadi Katz}

\author[M.~Katz]{Mikhail G. Katz$^{*}$}

\address{Department of Mathematics, Bar Ilan University, Ramat Gan
52900 Israel} \email{katzmik ``at'' macs.biu.ac.il}

\thanks{$^{*}$Supported by the Israel Science Foundation (grants
no.~84/03 and 1294/06) and the BSF (grant 2006393)}

\subjclass%[2000]
{Primary 53C23; %Global topological methods (a la Gromov) 
Secondary 30F10, %Compact Riemann surfaces and uniformisation 
58J60%Relations with special manifold structures (Riemannian, F) 
}


\date{\today}

\begin{abstract}
Given a hyperelliptic Klein surface, we construct companion Klein
bottles.  Bavard's short loops on companion bottles are studied in
relation to the surface to improve an inequality of Gromov's in
systolic geometry.
\end{abstract}

\maketitle 

\tableofcontents

\section{Introduction}

Systolic inequalities for surfaces compare length and area, and can
therefore be thought of as ``opposite'' isoperimetric inequalities.
The study of such inequalities was initiated by C. Loewner in '49 when
he proved his torus inequality for~$\T^2$.  In higher dimensions, we
have M.~Gromov's deep result \cite{Gr1} on the existence of a
universal upper bound for the systole in terms of the volume of an
essential manifold.  A promising research direction, initiated by
L.~Guth \cite{Gu, Gu09} is a search for a proof of Gromov's bound
without resorting to filling invariants as in \cite{Gr1}.

In dimension~$2$, the focus has been, on the one hand, on obtaining
near-optimal asymptotic results in terms of the genus \cite{KS2, KSV},
and on the other, on obtaining realistic estimates in cases of low
genus \cite{KS1, HKK}.  One goal has been to determine whether all
aspherical surfaces satisfy Loewner's bound, a question that is still
open in general.  It was resolved in the affirmative for genus~$2$ in
\cite{KS1}.  An older optimal inequality of C.~Bavard \cite{Bav1} for
the Klein bottle~$K$ is stronger than Loewner's bound.  Other than
these results, no bound is available that's stronger than Gromov's
general estimate \cite{Gr1} for aspherical surfaces:
\begin{equation}
\label{11b}
\sys^2 \leq \frac{4}{3} \area,
\end{equation}
even for the genus~$3$ surface.  There are at most~$17$ genera where
Loewner's bound could be violated \cite{KS2}, including genus~$3$.

As Gromov points out in \cite{Gr1}, the~$\frac{4}{3}$ bound
\eqref{11b} is actually {\em optimal\/} in the class of Finsler
metrics.  Therefore any further improvement is not likely to result
from a simple application of the coarea formula.  One can legitimately
ask whether any improvement is in fact possible, of course in the
framework of Riemannian metrics.

Our purpose in the present article is to furnish such an improvement
in the case of the surface
\[
K\# \rp^2 = \T^2 \# \rp^2 = 3\rp^2
\]
of Euler characteristic~$-1$, as follows.

\begin{theorem}
The surface~$3\rp^2$ satisfies the bound
\[
\sys^2 \leq 1.333 \; \area
\]
(see Theorem~\ref{15b} below).
\end{theorem}

Note the absence of an ellisis following ``333'', making our estimate
an improvement on Gromov's~$\frac{4}{3}$ bound \eqref{11b}.  Our proof
exploits a variety of techniques ranging from hyperellipticity to the
coarea formula and cutting and pasting.  An inspection of the proof
reveals that all the estimates are very tight and only produce an
improvement in the fourth decimal place, a fact the present writer has
no explanation for other than a remarkable coincidence.

Note that the current best upper bound for the systole only differs by
about~$30\%$ from the Silhol-Parlier hyperblolic example (see
Section~\ref{fourteen}).

\section{Hyperellipticity}

Complex-analytic information can be applied to the study of Klein
surfaces and their Riemannian geometry.  The argument involves
dragging a short loop across, from the extreme right to the extreme
left of the diagram
\begin{equation*}
3\RP^2 \leftarrow \Sigma_2 \to S^2 \leftarrow \T_{a,b} \to K .
\end{equation*}
(see \eqref{53} for details).  The quadratic equation
\begin{equation*}
y^2=p
\end{equation*}
over~$\C$ is well known to possess two distinct solutions for every
$p\not=0$, and a unique solution for~$p=0$.  Similarly, the locus
(solution set) of the equation
\begin{equation}
\label{11}
y^2= p(x)
\end{equation}
for~$(x,y)\in \C^2$, where~$p(x)$ is a (generic) polynomial of even
degree
\begin{equation*}
2g+2,
\end{equation*}
defines a Riemann surface which is a branched two-sheeted cover
of~$\C$.  Such a cover is constructed by projection to
the~$x$-coordinate.  The branching locus corresponds to the roots
of~$p(x)$.  It is known that there exists a unique smooth closed
Riemann surface~$\Sigma_g$ naturally associated with \eqref{11},
sometimes called the {\em smooth completion\/} of the affine
surface~\eqref{11}, together with holomorphic map
\begin{equation}
\label{12}
P_x: \Sigma_g\to \hat \C=S^2
\end{equation}
extending the projection to the~$x$-coordinate.  By the
Riemann-Hurwitz formula, the genus of the smooth completion is~$g$.
All such surfaces are hyperelliptic by construction, where the
hyperelliptic involution
\begin{equation*}
J: \Sigma_g \to \Sigma_g
\end{equation*}
flips the two sheets of the double cover of~$S^2$.

\section{A pair of involutions}

A hyperelliptic closed Riemann surface~$\Sigma_g$ admitting an
orientation-reversing (antiholomorphic) involution~$\tau$ can always
be reduced to the form \eqref{11} where~$p(x)$ is a polynomial all of
whose coefficients are real, where the involution
\begin{equation*}
\tau: \Sigma_g \to \Sigma_g
\end{equation*}
restricts to complex conjugation on the affine part of the surface
in~$\C^2$, namely~$\tau(x,y)=(\bar x, \bar y)$.  

The special case of a fixed point-free involution~$\tau$ can be
represented as the locus of the equation
\begin{equation}
\label{21}
-y^2 = \prod_i (x-x_i)(x-\bar x_i),
\end{equation}
where~$x_i\in \C\setminus \R$ for all~$i$.  Here the minus sign on the
left hand side ensures the absence of real solutions, and therefore
the fixed point-freedom of~$\tau$.

By the uniqueness of the hyperelliptic involution, we have the
commutation relation
\begin{equation*}
\tau \circ J = J \circ \tau.
\end{equation*}

\section{Klein surfaces and bottles}

A Klein surface is a non-orientable closed surface.  Such a surface
can be thought of as an antipodal quotient~$\Sigma_g/\tau$ of an
orientable surface by a fixed-point free, orientation-reversing
involution~$\tau$.  It is known that~$(\Sigma_g,\tau)$ can be thought
of as a real surface.  A Klein surface is homeomorphic to the
connected sum
\begin{equation*}
\RP^2\#\RP^2\#\cdots\#\RP^2=n\RP^2
\end{equation*}
of~$n$ copies of the real projective plane.  The case~$n=2$
corresponds to the Klein bottle~$K=2\RP^2$.  The orientable double
cover of the Klein bottle is a torus~$\T^2$.  The Klein bottle~$K$ can
be thought of as a pair~$(\T^2, \tau)$, or more precisely the quotient
\begin{equation*}
\T^2/\{1,\tau\}
\end{equation*}
where~$\tau$ is a fixed point free, orientation-reversing involution.
We will sometimes use the abbreviated notation
\begin{equation*}
K=2\RP^2 = \T^2/\tau,
\end{equation*}
and refer to it as the antipodal quotient.  In the case~$n=3$, we
obtain the surface~$3\RP^2$ of Euler characteristic~$-1$, whose
orientable double cover is the genus~$2$ surface~$\Sigma_2$.  Thus,
the surface~$3\RP^2$ can be thought of as the pair~$(\Sigma_2, \tau)$,
or more precisely the quotient
\begin{equation*}
3\RP^2= \Sigma_2/\{1,\tau\},
\end{equation*}
where~$\tau$ is a fixed point free, orientation-reversing involution.
We will sometimes use the abbreviated notation
\begin{equation*}
3\RP^2 = \Sigma_2/\tau,
\end{equation*}
and refer to it as the antipodal quotient.

\section{?-sided loops}

A loop on a surface is called~$2$-sided if its tubular neighborhood is
homeomorphic to an annulus, and~$1$-sided if it is homeomorphic to a
Mobius strip.  A one-sided loop
\begin{equation*}
\gamma\subset 3\RP^2
\end{equation*}
lifts to a path on~$(\Sigma_2, \tau)$ connecting a pair of points
which form an orbit of the involution~$\tau$.  In other words, the
inverse image of~$\gamma$ under the double cover~$\Sigma_2 \to 3\RP^2$
is a circle (i.e. has a single connected component homeomorphic to a
circle).  Meanwhile, a two-sided loop
\begin{equation*}
\delta \subset 3\RP^2
\end{equation*}
lifts to a closed curve on~$(\Sigma_2, \tau)$.  In other words, the
inverse image of~$\delta$ under the double cover~$\Sigma_2 \to 3\RP^2$
has a pair of connected components (circles).

\section{Companionship}

Given a real Riemann surface~$(\Sigma_g, \tau)$, consider the
presentation \eqref{21} with~$p$ real.  We can write the roots of~$p$
as a collection of pairs~$(a, \bar a)$.  In the genus~$2$ case, we
have three pairs~$(a, \bar a, b, \bar b, c, \bar c)$.  Thus the affine
form of the surface is the locus of the equation
\begin{equation}
\label{51}
-y^2 = (x-a)(x-\bar a) (x-b)(x-\bar b) (x-c)(x-\bar c)
\end{equation}
in~$\C^2$.  Choosing two conjugate pairs, for instance~$(a, \bar a, b,
\bar b)$, we can construct a companion surface
\begin{equation}
\label{52}
-y^2 = (x-a)(x-\bar a) (x-b)(x-\bar b),
\end{equation}
By the Riemann-Hurwitz formula, we have~$g=1$ and therefore the
(smooth completion of the) companion surface is a torus.  We will
denote it
\begin{equation*}
\T_{a,b}.
\end{equation*}
By construction, its set of zeros is~$\tau$-invariant.  In other
words, the (affine part in~$\C^2$ of the) torus is invariant under the
action of complex conjugation.  Thus, the surface~$\T_{a,b}/\tau$ is a
Klein bottle~$K$, which we will refer to as a companion Klein bottle
of the original Klein surface~$3\RP^2=\Sigma_2/\tau$, namely the
antipodal quotient of~$\eqref{51}$.

The maps constructed so far can be represented by the following
diagram of homomorphisms (note that two out of four point lefttward):
\begin{equation}
\label{53}
3\RP^2 \leftarrow \Sigma_2 \to S^2 \leftarrow \T_{a,b} \to K .
\end{equation}

\section{Equator}
\label{six}

Complex conjugation on~$\hat \C = S^2$ fixes a circle called the
equator, which could be denoted
\begin{equation*}
\hat\R \subset \hat \C.
\end{equation*}
The inverse image of the equator~$\hat \R$ under the double cover
$\Sigma_2 \to S^2$ is also a circle.  We will refer to it as the
equatorial circle, or equator, of~$\Sigma_2$.  

\begin{lemma}
The equator of~$\Sigma_2$ is the fixed point set of the composed
involution~$\tau \circ J=J\circ \tau$.  The equator is invariant under
the action of~$\tau$.
\end{lemma}

The action of~$\tau$ on the equator of~$\Sigma_2$ is fixed point-free
and thus can be thought of as a rotation by~$\pi$.

In the case of the double cover~$\T_{a,b} \to S^2$, the inverse image
of the equator is a pair of disjoint circles.  The involution~$\tau$
acts on the torus by switching the two circles.  Note that the
involution~$\tau \circ J$ fixes both circles pointwise.

\section{Area and systole of a Klein surface}

So far we have mostly dealt with conformal information independent of
the metric.  In this section we begin to consider metric-dependent
invariants.

\begin{lemma}
Given a~$J$-invariant metric on the Klein surface~$3\RP^2$, we have
the following relations among the areas of the surfaces appearing in
\eqref{53}:
\begin{equation*}
\area(3\RP^2) = \area (S^2) = \area (K),
\end{equation*}
as well as
\begin{equation*}
\area(\Sigma_2) = \area (\T_{a,b}) = 2 \area (3\RP^2).
\end{equation*}
\end{lemma}

The systole, denoted ``sys'', of a space is the least
length of a loop which cannot be contracted to a point in the space.
Given a metric on a Klein surface
\begin{equation*}
n\RP^2=\Sigma_g/\tau,
\end{equation*}
we consider the natural pullback metric on its orientable double
cover, denoted~$\Sigma_g$.  This metric is invariant under the
involution~$\tau$.  We define the least displacement invariant
``$\disp$'' by setting
\begin{equation}
\label{61b}
\disp(\tau)= \min \left\{ \dist(x,\tau(x)) \mid x\in \Sigma_g
\right\}.
\end{equation}
The systole of~$\sys(n\RP^2)$ can be expressed as the least of the
following two quantities:
\begin{equation*}
\sys(n\RP^2) = \min \left\{ \sys(\Sigma_g), \disp(\tau) \right\},
\end{equation*}
where~$\disp$ is the least displacement defined in \eqref{61b}.

\section{Systolic estimates}

Given a Riemannian metric on a Klein surface~$n\RP^2$, we are
interested in obtaining upper bounds for its systolic ratio
\begin{equation}
\label{61}
\frac{\sys^2}{\area},
\end{equation}
where ``sys'' is its systole.  Recall that the following four
properties of a closed surface~$\Sigma$ are equivalent:

\begin{enumerate}
\item
$\Sigma$ is aspherical;
\item
the fundamental group of~$\Sigma$ is infinite;
\item
the Euler characteristic of~$\Sigma$ is non-positive;
\item
$\Sigma$ is not homeomorphic to either~$S^2$ or~$\RP^2$.
\end{enumerate}

The following conjecture has been discussed in the systolic
literature, see \cite{SGT}.

\begin{conjecture}
Every aspherical surface satisfies Loewner's bound
\begin{equation}
\label{loew}
\frac{\sys^2}{\area} \leq \frac{2}{\sqrt{3}}.
\end{equation}
\end{conjecture}
The conjecture was proved for the torus by C. Loewner in '49 and for
the Klein bottle by C. Bavard \cite{Bav1}.  M. Gromov \cite{Gr1}
proved an asymptotic estimate which implies that every orientable
surface of genus greater than~$50$ satisfies Loewner's bound.  This
was extended to orientable surfaces of genus at least~$20$ by M. Katz
and S. Sabourau \cite{KS2}, and for the genus 2 surface in \cite{KS1}.

If the orientable double cover of the Klein surface is hyperelliptic,
then its metric can be averaged by the hyperelliptic involution~$J$ so
as to improve (i.e. increase) its systolic ratio \eqref{61}.  This
point was discussed in detail in \cite{BCIK1}.  Similarly, the ratio 
\begin{equation*}
\frac{\disp(\tau)^2}{\area}
\end{equation*}
defined in \eqref{61b} is increased by averaging.  Thus we may assume
without loss of generality that the metric on the orientable double
cover is already~$J$-invariant.

Consider a systolic loop on~$n\RP^2$, namely a loop of least length
which cannot be contracted to a point in~$n\RP^2$.  The loop is either
a 1-sided loop~$\gamma$, or a~$2$-sided loop~$\delta$.

\section{Outline of the argument}

Assume that we are in the situation of a~$2$-sided loop~$\delta
\subset K$.  Since~$\delta$ lifts to a closed loop~$\tilde \delta$ on
the orientable double cover~$\Sigma_g$, we have
\begin{equation*}
\sys(n\RP^2) = \sys(\Sigma_g).
\end{equation*}
In the special situation~$n=3$ and~$g=2$, we proceed as follows.
Consider the real model \eqref{51} of~$\Sigma_g$, and its three
companion tori of type~\eqref{52}.

For each companion torus, we pass to the quotient Klein bottle, and
find a systolic loop satisfying Bavard's inequality (which is stronger
than Loewner's bound).  We thus obtain three loops~$\delta_{a,b},
\delta_{b,c}, \delta_{a,c}$, and by our assumption the loops lift to
loops~$\tilde\delta_{a,b}, \tilde\delta_{b,c}, \tilde\delta_{a,c}$ on
the torus.  Each loop~$\tilde\delta$ projected to the sphere~$S^2$
defines a partition of the set of six branch points.

We consider the corresponding partitions of the set of six branch
points on~$S^2$.  If the three partitions are not identical, then the
corresponding loops~$\tilde \delta$ can be rearranged by cutting and
pasting (see next section), to give a loop on the Klein surface which
satisfies Bavard's inequality.

In the remaining case, all three loops define identical partitions of
the six branch points of~$\Sigma_g$.  This case will be handled in
Section~\ref{fourteen}.

\section{Cut and paste technique}

The argument in \cite{KS1} can be summarized as follows.

\begin{lemma}
Given a genus~$2$ surface of unit area, one can find a pair of
non-homotopic loops~$\ell_1$ and~$\ell_2$, of combined length
\begin{equation*}
|\ell_1| + |\ell_2| \leq 2 C_{{\rm Loewner}},
\end{equation*}
where
\begin{equation*}
C_{{\rm Loewner}}=\sqrt{\frac{2}{\sqrt{3}}}
\end{equation*}
(see \eqref{loew} above).  In particular, the shorter of the two is a
Loewner loop on the surface.
\end{lemma}

The proof can be summarized as follows.  We exploit hyperellipticity
to construct a pair of companion tori.  We then apply Loewner's torus
inequality to find Loewner loops on each of the two tori.  Finally, we
apply a cut and paste technique to rearrange segments of the two loops
into a pair of loops that lift to the genus~$2$ surface.

Consider a systolic loop~$\gamma \subset K$ of a Klein bottle~$K$.  A
systolic loop is necessarily simple.  It satisfies Bavard's inequality
\cite{Bav1}
\begin{equation*}
|\gamma| \leq C_{{\rm Bavard}} \sqrt{\area(K)},
\end{equation*}
where~$C_{{\rm Bavard}}= \sqrt{\frac{\pi}{\sqrt{8}}}$.  Note
that~$C_{{\rm Bavard}} < C_{{\rm Loewner}}$ (Bavard's inequality gives
a stronger bound than Loewner's).

\begin{definition}
A loop satisfying Bavard's inequality will be called a Bavardian loop.
\end{definition}

\section{Proof:~$1$-sided systolic loop}

We now describe a procedure for constructing short loops on a Klein
surface.  Given a genus~$2$ surface \eqref{51}, consider a companion
Klein bottle
\begin{equation*}
K_{a,b} = \T_{a,b}/\tau.
\end{equation*}
Consider a systolic loop on~$K_{a,b}$.  There are two cases to
consider, according to the sidedness of the systolic loop.  In this
section, we consider the case when the systolic loop~$\gamma$ is
one-sided, which turns out to be the easier of the two.  Then~$\gamma$
lifts to a path connecting a pair of opposite points of the torus.
Let~$\tilde \gamma$ be its connected double cover on the
torus~$\T_{a,b}$.  

\begin{lemma}
The loop~$\tilde\gamma$ contains a real point~$p$ such that the points
$p$ and~$\tau(p)$ decompose~$\tilde\gamma$ into a pair of paths
\begin{equation*}
\tilde\gamma= \gamma_+ \cup \gamma_- 
\end{equation*}
such that each of~$\gamma_+, \gamma_-$ projects to a loop on~$S^2$.
\end{lemma}

\begin{proof}
Note that~$\tilde \gamma$ is invariant under the action of~$\tau$ on
the torus.  The loop~$\tilde \gamma$ projects to a loop
\begin{equation*}
\tilde \gamma_0 \subset S^2.
\end{equation*}
The connected loop~$\tilde \gamma_0$ is invariant under complex
conjugation.  Therefore it must meet the equator (see
Section~\ref{six}).  Here the equator is the closure of~$\R$ in~$\hat
\C$.  A point~$p_0\in \hat \R$ where~$\tilde \gamma_0$ meets the
equator corresponds to a real point~$p\in \tilde \gamma$ such that~$p$
and~$\tau(p)$ have the same image in the sphere, namely~$p_0$.
Thus~$\tilde \gamma_0$ necessarily has a self-intersection,
unlike~$\gamma$ itself.

It may be helpful to think of~$\tilde \gamma_0$ as a figure-eight
loop, with its midpoint on the equator.  In fact, the original
loop~$\gamma$ on the bottle lifts to a path~$\gamma_+$ joining~$p$
and~$\tau(p)$, which then projects to half the loop~$\tilde\gamma_0$,
forming one of the hoops of the figure-eight.  If we let~$\gamma_-=
\tau(\gamma_+)$, we can write
\begin{equation*}
\tilde\gamma= \gamma_+ \cup \gamma_- .
\end{equation*}
Thus the loop~$\tilde\gamma_0 \subset S^2$ is the union of two loops
\begin{equation*}
\tilde\gamma_0= P_x(\gamma_+) \cup P_x(\gamma_-),
\end{equation*}
where~$P_x: \Sigma_2 \to S^2$ is defined (on the affine part) by the
projection to the~$x$-coordinate, see \eqref{12}.
\end{proof}

Next, we would like to transplant~$\gamma_+$ to the Klein surface
$3\RP^2$.  The loop~$P_x(\gamma_+)\subset S^2$ lifts to a path
\begin{equation*}
\tilde{\tilde\gamma} \subset \Sigma_2
\end{equation*}
which may or may not close up, depending on the position of the third
pair~$(c,\bar c)$ of branch points of~$\Sigma_2\to S^2$.  If
$\tilde{\tilde\gamma}$ is already a loop, then it projects to a
Bavardian loop on the Klein surface~$3\RP^2=\Sigma/\tau$.  In the
remaining case, the path~$\tilde{\tilde\gamma}$ connects a pair of
opposite points~$\tilde p, \tau(\tilde p)$ on the surface~$\Sigma_2$,
where
\begin{equation*}
P_x(\tilde p)= P_x(\tau(\tilde p))= p_0.
\end{equation*}
Therefore~$\tilde{\tilde\gamma}$ projects to a Bavardian loop in this
case, as well.

\section{Proof continued:~$2$-sided systolic loop}

It remains to consider the case when each systolic loop~$\delta$ of
each of the three companion Klein bottles
\begin{equation*}
K_{a,b}, K_{a,c}, K_{b,c}
\end{equation*}
of the genus~$2$ surface \eqref{51} is~$2$-sided.  Thus each of these
Bavardian loops~$\delta_{a,b},\delta_{a,c},\delta_{b,c}$ lifts to a
closed loop on the corresponding orientable double cover.  The
resulting loop on the corresponding torus will be denoted
$\tilde\delta$.  The projection
\begin{equation*}
P_x(\tilde\delta) \subset S^2
\end{equation*}
will be denoted~$\delta_0 = P_x(\tilde\delta)$.

\begin{lemma}
If~$\delta_0$ meets the equator, then~$3\RP^2$ contains a Bavardian
loop.
\end{lemma}

\begin{proof}
Let
\begin{equation*}
p_0 \in \delta_0 \cap \hat \R.
\end{equation*}
Let~$p, \tau(p) \in \Sigma_2$ be the points above it in~$\Sigma_2$.
We lift the path~$\delta_0$ to a path~$\delta_+ \subset \Sigma_2$
starting at~$p$.  If~$\delta_+$ closes up, its projection to~$3\RP^2$
is the desired Bavardian loop.  Otherwise, the path~$\delta_+$
connects~$p$ to~$\tau(p)$.  In this case as well,~$\delta_+$ projects
to a Bavardian loop on~$3\RP^2=\Sigma_2/\tau$.
\end{proof}

Thus we may assume that the~$\delta_0$ does not meet the equator of
$S^2$.

\begin{lemma}
Let~$\delta$ be a systolic loop on a companion torus~$\T_{a,b}$, and
assume~$\delta_0$ does not meet the equator.  Then~$\delta_0$ is a
simple loop.
\end{lemma}

\begin{proof}
By hypothesis, the loop~$\delta_0$ lies in a hemisphere.  The typical
case to keep in mind of a non-simple loop is a figure-eight.  Without
loss of generality, we may assume that~$\delta_0$ has odd winding
number with respect to the ramification points
\begin{equation*}
a,b\in S^2.
\end{equation*}
We will think of the curve~$\delta_0$ as defining a connected
graph~$\Delta \subset \R^2$ in a plane.  The vertices of the graph are
the self-intersection points of~$\delta_0$.  Each vertex necessarily
has valence~$4$.  By adding the bounded ``faces'', we obtain a ``fat''
graph, denoted~$\Delta_{{\rm fat}}$ (the typical example is the
interior of the figure-eight).  More precisely, the complement~$\R^2
\setminus \Delta$ has a unique {\em unbounded\/} connected component,
denoted~$E \subset \R^2 \setminus \Delta$.  Its complement in the
plane, denoted
\begin{equation*}
\Delta_{{\rm fat}}= \R^2 \setminus E,
\end{equation*}
contains the graph~$\Delta$ as well as its bounded ``faces''.  The
important point is that both branch points~$a,b$ of the double cover
$P_x: \T_{a,b}^2 \to S^2$ must lie inside the connected
region~$\Delta_{{\rm fat}}$:
\begin{equation*}
a,b\in \Delta_{{\rm fat}}.
\end{equation*}
The boundary of~$\Delta_{{\rm fat}}$ can be parametrized by a closed
curve~$\ell$, thought of the boundary of the outside
component~$E\subset \R^2$ so as to define an orientation on~$\ell$ (in
the case of the figure-eight loop, this results in reversing the
orientation on one of the hoops of the figure-eight).  Note that~$\ell
\subset \Delta$ is a subgraph.  Since both branch points lie inside,
the loop~$\ell$ lifts to a loop on the torus which cannot be
contracted to a point.  If~$\delta_0$ is not simple, then~$\ell$ must
contain ``corners'' (as in the case of a figure
eight-shaped~$\delta_0$) and can therefore be shortened, contradicting
the hypothesis that~$\delta$ is a systolic loop.
\end{proof}

We may thus assume that the simple loop~$\delta_0 \subset S^2$
separates the six points~$a,\bar a, b, \bar b, c, \bar c$ into two
triplets~$(a,b,c)$ and~$(\bar a, \bar b, \bar c)$.  Hence its
connected double cover in~$\Sigma_2$ is isotopic to the equatorial
circle of~$\Sigma_2$.  The latter is a double cover of the equator
of~$S^2$ (see Section~\ref{six}).

\section{Hyperellipticity for Klein surfaces}

It may be helpful to think of the surface~$3\RP^2$ as a double cover
of the northern hemisphere of~$S^2$, with the equator included.  The
cover is branched along the equator as well as at~$3$ additional
branch points, namely the points~$a,b,c$ of the (standard)
hyperelliptic cover~$\Sigma_2 \to S^2$.  Note that the three remaining
branch points~$\bar a, \bar b, \bar c$ are mapped to~$a,b,c$ by the
involution~$\tau$.

A horizontal circle on~$S^2$ with small positive latitude (south of
all three branch points) is double-covered by a circle on~$3\RP^2$
which can be thought of as the boundary of a Mobius strip, the central
circle of which is the equator.  The simple loop double covering
$\delta_0$ separates the surface into two surfaces with boundary: a
torus with circular boundary, and a Mobius strip.  The results of the
previous section are summarized in the diagram of
Figure~\ref{hemisphere}.

\begin{figure}%[ht]
\includegraphics[height=3in]{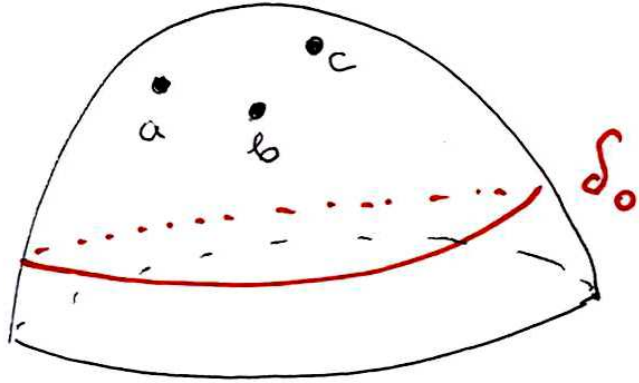}
\caption{Klein surface as a hemispherical double cover}
\label{hemisphere}
\end{figure}

\begin{proposition}
\label{141}
The Klein surface~$3\RP^2$ either contains a Bavardian loop, or admits
a simple loop which separates it into a torus with a disk
removed~$\Sigma_{1,1}$ and a Mobius strip~$\Mob$:
\begin{equation*}
3\RP^2 = \Sigma_{1,1} \cup_{S^1}^{\phantom{I}} \Mob
\end{equation*}
corresponding to a topological decomposition
\begin{equation*}
3\RP^2 = \T^2 \# \RP^2,
\end{equation*}
such that moreover
\begin{itemize}
\item
the torus with a disk removed~$\Sigma_{1,1}$ contains the three
isolated branch points;
\item
the Mobius strip contains the equatorial circle;
\item
the separating loop is~$J$-invariant;
\item
the separating loop is of length at most~$2C_{{\rm Bavard}}$;
\item
the separating loop double-covers a loop
\begin{equation*}
\delta_0\subset S^2
\end{equation*}
which lifts to a systolic loop on a companion torus.
\end{itemize}
\end{proposition}

\section{Improving Gromov's~$3/4$ bound}
\label{fourteen}

Gromov's general~$3/4$ bound for aspherical surfaces,
\begin{equation*}
\area \geq \frac{3}{4} \sys^2,
\end{equation*}
appeared in \cite[Corollary~5.2.B]{Gr1}.  We would like to improve the
bound in our case of~$3\RP^2$.  The bound can be written as
\begin{equation*}
\frac{\sys^2}{\area} \leq 1.3333\ldots
\end{equation*}

\begin{theorem}
\label{15b}
The bound
\begin{equation}
\label{151b}
\frac{\sys^2}{\area} \leq 1.333
\end{equation}
is satisfied by every metric on the Klein surface~$3\RP^2$.
\end{theorem}

Note that the highest systole of a hyperbolic metric on~$3\RP^2$ was
identified by H.~Parlier \cite{Par} to be~$\arc\!\cosh
\frac{5+\sqrt{17}}{2} = 2.19\ldots$ resulting in a systolic ratio of
$.76\ldots$.  The orientable double cover of this hyperbolic Klein
surface is the Parlier-Silhol curve of genus~$2$ \cite{Par,Si}.

\begin{proof}
A partition of~$3\RP^2$ into a torus with a disk removed and a Mobius
strip was constructed in Proposition~\ref{141}.  We normalize~$3\RP^2$
to unit area.  We will find a short loop on the torus with a disk
removed if its area is at most~$.676$, and on the Mobius strip if its
area is at most~$.324$.

Let~$\alpha = 1.333$.  To prove Theorem~\ref{15b}, we need to locate
an essential loop of length at most
\[
\beta = \sqrt{\alpha} \approx 1.15456 .
\]
If there is no such loop, the distance from each branch point to the
equator must be at least~$\tfrac{1}{2}\beta$.  We apply the
coarea formula to the distance function from the equator.  Since the
lift of~$\delta_0$ to the torus~$\T_{a,b}$ is a systolic loop by
Proposition~\ref{141}, we obtain
\begin{equation*}
2 |\delta_0| \tfrac{1}{2} \beta \leq 1,
\end{equation*}
as the torus is normalized to unit area.  Hence~$|\delta_0| \leq
\beta^{-1}$.  The argument with attaching a hemisphere that we will
present in Lemma~\ref{163} produces an estimate that becomes far more
powerful as~$|\delta_0|$ decreases.  Therefore there is no loss of
generality in assuming that equality takes place:
\begin{equation*}
|\delta_0| = \beta^{-1} \approx .86613 .
\end{equation*}
The theorem now results from the two lemmas below.
\end{proof}

\begin{lemma}
If the Mobius strip has area at most~$| \Mob | \leq .324$ then it
contains an essential loop of square-length less than~$\alpha=1.333$.
\end{lemma}

\begin{proof}
Let~$h$ be the least distance from a point of~$\delta_0$ to the
equator.  By the coarea formula applied to the distance function from
the equator, we obtain
\begin{equation}
\label{142} 
|\Mob| \geq 2 h |\delta_0| .
\end{equation}
Hence
\begin{equation*}
h \leq \frac{|\Mob|}{2|\delta_0|} \approx .18704 .
\end{equation*}
Connecting~$\delta_0$ to a nearest point of the equator by a pair of
paths of length~$h$, we obtain an essential loop on~$3\RP^2$ of length
at most
\begin{equation*}
|\delta_0| + 2h = \beta^{-1} + 2\frac{|\Mob|}{2|\delta_0|} =
\beta^{-1} + \beta |\Mob| \approx 1.26 ,
\end{equation*}
falling {\em short\/} (or, rather, {\em long}) of the required upper
bound of~$\beta \approx 1.15456$.  We will improve the estimate
\eqref{142} as follows.  If
\[
h \leq \frac{1}{2} ( \beta - |\delta_0|) \approx \frac{1}{2} (1.15456
- .86613) \approx .14422 ,
\]
then we obtain a short loop and prove the theorem.  Thus we may assume
that~$h\geq .14422$.

A level curve at distance~$x$ from the equator must have length at
least~$\beta -2x$ to avoid the creation of a short loop.
Hence~$|\Mob|$ is bounded below by {\em twice\/} the area of a trapeze
of altitude~$h$, larger base~$\beta$, and smaller base~$\delta_0$:
\[
|\Mob| \geq h (\beta + \delta_0) = .14422 \; ( 1.15456 + .86613)
 \approx .29 ,
\]
which falls short of the required estimate~$.324$.

Blatter \cite{Bl, Bl2} and Sakai
\cite{Sak} provide a lower bound equal to the half the area of a belt
formed by an~$h$-neighborhood of the equator of a suitable sphere of
constant curvature.  Here the equator has length~$2\beta$ while the
its antipodal quotient is a Mobius strip of systole~$\beta$.  The
radius of such a sphere is~$r = \frac{2\beta}{2\pi}=
\frac{\beta}{\pi}$.  Hence the subtending angle~$\gamma$ of the
northern half of the belt satisfies
\[
\gamma = \frac{h}{r} = \frac{h\pi}{\beta} \approx .39241 .
\]
(here we use the values~$\beta \approx 1.15456$ and~$|\delta_0| =
\beta^{-1} \approx .86613$).  The height function is the moment map
(Archimedes's theorem), and hence the area of the belt is proportional
to
\[
\sin \gamma \approx .38242 .
\]
The area of the corresponding region on the unit sphere is~$4\pi \sin
\gamma$.  Hence the area of the spherical belt is
\[
4 \pi r^2 \sin \gamma = \frac{4 \pi \beta^2 \sin \gamma}{\pi^2},
\]
which after quotienting by the antipodal map yields a lower bound
\[
|\Mob| \geq \frac{2 \beta^2 \sin \gamma}{\pi} = \frac{2 \alpha \sin
  \gamma}{\pi} = \frac{2 (1.333) \sin \gamma}{\pi} \approx .32453 >
  .324 ,
\]
proving the lemma.
\end{proof}

\begin{lemma}
\label{163}
If the torus with a disk removed has area at most~$.676$, then it
contains an essential loop of square-length less than~$\alpha= 1.333$.
\end{lemma}

\begin{proof}
The separating loop is a circle of radius
\begin{equation*}
r = \frac{2|\delta_0|}{2\pi} = \frac{|\delta_0|}{\pi}.
\end{equation*}
The area of a hemisphere based on such a circle is
\begin{equation*}
2\pi r^2 = \frac{2 |\delta_0|^2}{\pi} = \frac{2}{\pi \alpha}.
\end{equation*}
Attaching the hemisphere to the torus with a disk removed produces a
torus of total area at most
\begin{equation*}
\frac{2}{ \pi \alpha} + .676.
\end{equation*}
Applying Loewner's bound \eqref{loew} to the resulting torus, we
obtain a systolic loop of square-length at most
\begin{equation*}
\frac{2}{\sqrt{3}} \left( \frac{2}{\pi \alpha} + .676 \right) \approx
1.33204 < 1.333,
\end{equation*}
proving the lemma and the theorem.
\end{proof}

\begin{remark}
Gromov points out at the bottom of page~49 in \cite{Gr1} that his
$3/4$ bound can be improved by 1\%; perhaps this estimate is what he
had in mind.
\end{remark}

\section{Acknowledgments}

We are grateful to C. Croke for helpful discussions.

\vfill\eject


\begin{thebibliography}{Ai}


\bibitem{e7} Bangert, V; Katz, M.; Shnider, S.; Weinberger, S.: $E_7$,
Wirtinger inequalities, Cayley~$4$-form, and homotopy.  {\em Duke
Math. J.} \textbf{146} ('09), no.~1, 35-70.  See arXiv:math.DG/0608006



\bibitem{BCIK1} Bangert, V; Croke, C.; Ivanov, S.; Katz, M.:
Filling area conjecture and ovalless real hyperelliptic surfaces.
{\em Geometric and Functional Analysis (GAFA)\/} \textbf{15} (2005)
no.~3, 577-597.  See arXiv:math.DG/0405583



\bibitem{Bav1} Bavard, C.: In\'egalit\'e isosystolique pour la
bouteille de Klein.  {\em Math. Ann.} \textbf{274} (1986), no.~3,
439--441.


\bibitem{Be6} Berger, M.: A panoramic view of Riemannian geometry.
Springer-Verlag, Berlin, 2003.

\bibitem{Be08} Berger, M.: What is... a Systole? {\em Notices of the
AMS\/} \textbf{55} (2008), no.~3, 374-376.


\bibitem{Bl} Blatter, C.: \"{U}ber Extremall\"{a}ngen auf
geschlossenen Fl\"{a}chen.  \textit{Comment.\ Math.\ Helv.}
\textbf{35} (1961), 153--168.


\bibitem{Bl2} Blatter, C.: Zur Riemannschen Geometrie im
Grossen auf dem M\"{o}bius\-band.  {\em Compositio Math.} \textbf{15}
(1961), 88--107.



\bibitem{Bru} Brunnbauer, M.: Homological invariance for
asymptotic invariants and systolic inequalities.  {\em Geometric and
Functional Analysis (GAFA)}, \textbf{18} ('08), no.~4, 1087--1117.
See arXiv:math.GT/0702789


\bibitem{Bru2} Brunnbauer, M.: Filling inequalities do not
depend on topology.  {\em J. Reine Angew. Math.} \textbf{624} (2008),
217--231.  See arXiv:0706.2790


\bibitem{Bru3} Brunnbauer M.: On manifolds satisfying stable
systolic inequalities.  {\em Math. Annalen\/} \textbf{342} ('08),
no.~4, 951--968.  See arXiv:0708.2589



\bibitem{DKR} Dranishnikov, A.; Katz, M.; Rudyak, Y.: Small
values of the Lusternik-Schni\-rel\-mann category for manifolds.  {\em
Geometry and Topology\/} \textbf{12} (2008), 1711-1727.  See
arXiv:0805.1527



\bibitem{Gr1} Gromov, M.: Filling Riemannian manifolds.  {\em
J. Diff. Geom.} \textbf{18} (1983), 1-147.


\bibitem{Gr2} Gromov, M.: Systoles and intersystolic
inequalities.  Actes de la Table Ronde de G\'{e}om\'{e}trie
Diff\'{e}rentielle (Luminy, 1992), 291--362, {\em S\'{e}min. Congr.},
\textbf{1}, Soc. Math. France, Paris, 1996. \newline\noindent
www.emis.de/journals/SC/1996/1/ps/smf\_sem-cong\_1\_291-362.ps.gz

\bibitem{Gr3} Gromov, M.: Metric structures for Riemannian and
non-Riemannian spaces.  {\em Progr. Math.} \textbf{152},
Birkh\"{a}user, Boston, 1999.


\bibitem{Gr4} Gromov, M.: Metric structures for Riemannian and
non-Riemannian spaces.  Based on the 1981 French original. With
appendices by M. Katz, P. Pansu and S. Semmes. Translated from the
French by Sean Michael Bates. Reprint of the 2001 English
edition. Modern Birkh\"auser Classics.  Birkh\"auser Boston, Inc.,
Boston, MA, 2007.



\bibitem{Gu} Guth, L.: Volumes of balls in large Riemannian manifolds.
%Available at 
arXiv:math.DG/0610212


\bibitem{Gu09} Guth, L.: Systolic inequalities and minimal
hypersurfaces.  
%See 
arXiv:0903.5299.


\bibitem{HKK} Horowitz, C.; Katz, Karin Usadi; Katz, M.: Loewner's
torus inequality with isosystolic defect.  {\em Journal of Geometric
Analysis}, to appear.  
%See 
arXiv:0803.0690



\bibitem{SGT} Katz, M.: Systolic geometry and topology.  With an
appendix by Jake P. Solomon.  {\em Mathematical Surveys and
Monographs}, \textbf{137}.  American Mathematical Society, Providence,
RI, 2007.

\bibitem{Ka4} Katz, M.: Systolic inequalities and Massey
products in simply-con\-nected manifolds.  {\em Israel J. Math.}
\textbf{164} (2008), 381-395.  arXiv:math.DG/0604012



\bibitem{KS2} Katz, M.; Sabourau, S.: Entropy of systolically
extremal surfaces and asymptotic bounds.  {\em Ergo. Th. Dynam. Sys.},
\textbf{25} (2005), no.~4, 1209-1220.  See arXiv:math.DG/0410312


\bibitem{KS1} Katz, M.; Sabourau, S.: Hyperelliptic surfaces are
Loewner.  {\em Proc. Amer. Math. Soc.} \textbf{134} (2006), no.~4,
1189-1195.  See arXiv:math.DG/0407009



\bibitem{KSV} Katz, M.; Schaps, M.; Vishne, U.: Logarithmic growth of
systole of arithmetic Riemann surfaces along congruence subgroups.
{\em J. Differential Geom.} \textbf{76} (2007), no.~3, 399-422.
Available at arXiv:math.DG/0505007



\bibitem{KSh} Katz, M.; Shnider, S.: Cayley 4-form comass and
triality isomorphisms.  {\em Israel J. Math.} (2009), to appear.  See
arXiv:0801.0283



\bibitem{Loe} Loewner, C.: Theory of continuous groups.  Notes
by H.~Flanders and M.~Protter.  {\em Mathematicians of Our Time\/}
\textbf{1}, The MIT Press, Cambridge, Mass.-London, 1971.



\bibitem{Par} Parlier, H.: Fixed-point free involutions on
Riemann surfaces.  {\em Israel J. Math.} \textbf{166} ('08), 297-311.
See arXiv:math.DG/0504109


\bibitem{Pu} Pu, P.M.: Some inequalities in certain
nonorientable Riemannian manifolds, {\it Pacific J. Math.\/}
\textbf{2} (1952), 55--71.











\bibitem{Sak} Sakai, T.: A proof of the isosystolic inequality
for the Klein bottle.  {\em Proc. Amer. Math. Soc.} \textbf{104}
(1988), no.~2, 589--590.


\bibitem{Si} Silhol, R.: On some one parameter families of genus
2 algebraic curves and half twists. {\em Comment. Math. Helv.}
\textbf{82} (2007), no.~2, 413--449.





\bibitem{Us} Usadi, Karin: A counterexample to the equivariant simple
loop conjecture.  {\em Proc. Amer. Math. Soc.} \textbf{118} (1993),
no.~1, 321--329.





\end{thebibliography}
\end{document}